\documentclass[11pt]{amsart}

\usepackage{amsmath,amsthm,amssymb,amscd,amsxtra,upref}
\usepackage{latexsym}
\usepackage[dvips]{graphics,epsfig}
\usepackage[colorlinks=true, pdfstartview=FitV, linkcolor=blue, citecolor=blue, urlcolor=blue]{hyperref}

\newtheorem{theorem}{Theorem}[section]
\newtheorem{lemma}[theorem]{Lemma}
\newtheorem{proposition}[theorem]{Proposition}
\newtheorem{corollary}[theorem]{Corollary}

\theoremstyle{remark}
\newtheorem{dfn}{Definition}[section]

\newtheorem{remark}{Remark}[section]

\newcommand{\mb}{\medskip\noindent}
\newcommand{\gb}{\bigskip\noindent}

\newcommand{\re}{\mathbb{R}}
\newcommand{\rn}{\mathbb{R}^n}

\newcommand{\M}{{\mathcal M}}
\newcommand{\R}{{\mathbb R}}
\newcommand{\pr}{{\mathbb P}}

\def\Xint#1{\mathchoice
   {\XXint\displaystyle\textstyle{#1}}%
   {\XXint\textstyle\scriptstyle{#1}}%
   {\XXint\scriptstyle\scriptscriptstyle{#1}}%
   {\XXint\scriptscriptstyle\scriptscriptstyle{#1}}%
   \!\int}
\def\XXint#1#2#3{{\setbox0=\hbox{$#1{#2#3}{\int}$}
     \vcenter{\hbox{$#2#3$}}\kern-.5\wd0}}

\def\aver#1{\Xint-_{#1}}

\begin{document}



\thanks{The author is supported by the ANR under the project AFoMEN no. 2011-JS01-001-01.}

\author{Fr\'ed\'eric Bernicot}
\address{Fr\'ed\'eric Bernicot, Laboratoire de Math\'ematiques Jean Leray \\ 2, Rue de la Houssini\`ere F-44322 Nantes Cedex 03, France.} \email{frederic.bernicot@univ-nantes.fr}

\title[Multi-frequency analysis]{Multi-frequency Calder\'on-Zygmund analysis and connexion to Bochner-Riesz multipliers}

\day=01 \month=04 \year=2013
\date{\today}

\begin{abstract} In this work, we describe several results exhibited during a talk at the {\it El Escorial 2012} conference.
We aim to pursue the development of a multi-frequency Calder\'on-Zygmund analysis introduced in \cite{NOT}. We set a definition of general multi-frequency Calder\'on-Zygmund operators. Unweighted estimates are obtained using the corresponding multi-frequency decomposition of \cite{NOT}. Involving a new kind of maximal sharp function, weighted estimates are obtained. 
\end{abstract}

\maketitle

The so-called Calder\'on-Zygmund theory and its ramifications have proved to be a powerful tool in many aspects of harmonic analysis and partial differential equations.
The main thrust of the theory is provided by 
\begin{itemize} 
 \item the Calder\'on-Zygmund decomposition, whose impact is deep and far-reaching. This decomposition is a crucial tool
to obtain weak type $(1,1)$ estimates and consequently $L^p$ bounds for a variety of operators;
 \item the use of the ``local'' oscillation $f - \left( \aver{Q}f \right)$ (for $Q$ a ball). These oscillations appear in the elementary functions of the ``bad part'' coming from the Calder\'on-Zygmund decomposition and in the definition of the maximal sharp function, which allows to get weighted estimates.
\end{itemize}

The oscillation $f - \left(\aver{Q} f\right)$ can be seen as the distance between the function $f$ and the set of constant functions on the ball $Q$, indeed the average is the best way to locally approximate the function by a constant. By this way, the constant function being associated to the frequency $0$, we understand how the classical Calder\'on-Zygmund theory is related to the frequency $0$.

As for example, well-known Calder\'on-Zygmund operators are the Fourier multipliers associated to a symbol $m$ satisfying H\"ormander's condition
$$ | \partial^\alpha m(\xi)| \lesssim |\xi|^{-|\alpha|}=d(\xi,0)^{-|\alpha|},$$
which encodes regularity assumption of the symbol relatively to the frequency $0$.

\mb
In this work, we are interested in the extension of this theory with respect to a collection of frequencies and we focus on sharp constants relatively to the number of the considered frequencies.

\gb Such questions naturally arise as soon as we work on a multi-frequency problem:
\begin{itemize}
 \item Uniform bounds for a Walsh model of the bilinear Hilbert transform (see \cite{OT} by Oberlin and Thiele);
 \item A variation norm variant of Carleson's theorem (see \cite{OSTTW} by Oberlin, Seeger, Tao, Thiele and Wright);
 \item Such a multi-frequency Calder\'on-Zygmund was introduced by Nazarov, Oberlin and Thiele in \cite{NOT} for proving a variation norm variant of a Bourgain's maximal inequality.
\end{itemize}

Similarly to the fact that a Fourier multiplier with a symbol satisfying H\"ormander's condition is a classical Calder\'on-Zygmund, we may extend this property to a collection of frequencies. More precisely, let $\Theta:=(\xi_1,...,\xi_N)$ be a collection of frequencies and consider a symbol $m$ verifying for all multi-indices $\alpha$
$$ | \partial^\alpha m(\xi)| \lesssim d(\xi,\Theta)^{-|\alpha|},$$
with $d(\xi,\Theta):= \min_{1\leq i \leq N} |\xi-\xi_i|$.
Such symbols give rise to Fourier multipliers, which should be the prototype of what we want to call {\it multi-frequency Calder\'on-Zygmund operators}.

\mb

In the $1$-dimensional setting with a collection of frequencies $\Theta:=(\xi_1,...,\xi_N)$ (assumed to be indexed by the increasing order $\xi_1< \xi_2< \cdots < \xi_N$), an example is given by the multi-frequency Hilbert transform which corresponds to the symbol
$$ m(\xi) = \left\{ \begin{array}{ll}
                     -1, & \xi<\xi_1 \\
                      (-1)^{j+1}, & \xi_j<\xi<\xi_{j+1} \\
                      (-1)^{N+1}, & \xi>\xi_N.
                    \end{array} \right. $$

Let us now detail a definition of ``multi-frequency Calder\'on-Zygmund'' operator:

\begin{dfn} Let $\Theta:=(\xi_1,...,\xi_N)$ be a collection of $N$ frequencies of $\R^n$. An $L^2$-bounded linear operator $T$ is said to be a Calder\'on-Zygmund operator relatively to $\Theta$ if there exist operators $(T_j)_{j=1,...,N}$ and kernels $(K_j)_{j=1,...,N}$ verifying 
 \begin{itemize}
  \item Decomposition: $T=\sum_{j=1}^N T_j$;
  \item Integral representation of $T_j$: for every function $f\in L^2$ compactly supported and $x\in \textrm{supp}(f)^{c}$, 
  $$ T_j(f)(x) = \int K_j(x,y) f(y);$$
  \item Regularity of the modulated kernels: for every $x\neq y$
$$ \sum_{j=1}^N \left|\nabla_{(x,y)}\ e^{i\xi_j\cdot(x-y)} K_j(x,y)\right| \lesssim |x-y|^{-n-1}.$$
 \end{itemize}
\end{dfn}

\begin{remark}
As usual, we can weaken the regularity assumption and just require an $\epsilon$-H\"older regularity on the modulated kernels.
\end{remark}

\begin{remark}
If the decomposition is assumed to be orthogonal (which means that for $i\neq j$, $T_iT_j^*=0$) then it follows that each operator $T_j$ is a modulated Calder\'on-Zygmund operator. Such a multi-frequency Calder\'on-Zygmund operator can also be pointwisely bounded by a sum of $N$ modulated (classical) Calder\'on-Zygmund operators and have the same boundedness properties with an implicit constant of order $N$. The aim is to study how this order can be improved using sharp estimates.
\end{remark}

We first obtain unweighted estimates for such operators:
\begin{theorem} \label{thm:boundedness} Let $\Theta$ be a collection of $N$ frequencies and $T$ an associated multi-frequency Calder\'on-Zygmund operator. Then 
\begin{itemize}
 \item for $p\in(1,\infty)$, $T$ is bounded on $L^p$ with
$$ \|T\|_{L^p \to L^p} \lesssim N^{\left|\frac{1}{p}-\frac{1}{2}\right|}.$$
 \item for $p=1$, $T$ is of weak-type $(1,1)$ with
$$ \|T\|_{L^1 \to L^{1,\infty}} \lesssim N^{\frac{1}{2}}.$$
\end{itemize}
\end{theorem}

This theorem relies on an adapted Calder\'on-Zygmund decomposition introduced in \cite{NOT} by Nazarov, Oberlin and Thiele. We point out that there the constant $N^{\frac{1}{2}}$ is shown to be optimal and this is the same for the previous weak-type estimate.

\bigskip
Concerning weighted estimates, it is well-known that linear Calder\'on-Zygmund operators are bounded on $L^p(\omega)$ for $p\in(1,\infty)$ and every weight $\omega$ belonging to the Muckenhoupt's class ${\mathbb A}_p$ (see Definitions \ref{def:1} and \ref{def:2} for more details about Muckenhoupt's class ${\mathbb A}_p$ and Reverse H\"older class $RH_s$). Similar properties are satisfied by the Hardy-Littlewood maximal operator and some other linear operators as Bochner-Riesz multipliers \cite{Vargas,CDL} or non-integral operators (like Riesz transforms) \cite{AM}. All these boundedness, obtained by using suitable Fefferman-Stein inequalities related to maximal sharp functions, involve weights belonging to the class ${\mathcal W}^p (p_0 , q_0 ) := {\mathbb A}_{\frac{p}{p_0}} \cap RH_{(\frac{q_0}{p})'}$ for some exponents $p_0<q_0$. 
\footnote{From \cite{JN}, we know that for $r,s>1$, 
$$ {\mathbb A}_{r} \cap RH_s = \left\{\omega, \omega^{s} \in{\mathbb A}_{1+s(r-1)}\right\},$$ so these classes of weights are equivalent to a class of powers of Muckenhoupt's weights.}

As a consequence, it seems that these classes of weights are well-adapted for proving boundedness of linear operators. Following this observation, we will consider a multi-frequency maximal sharp function, in order to prove weighted estimates for our multi-frequency operators:

\begin{theorem} \label{thm:weight} Let $\Theta$ be a collection of $N$ frequencies. For $p\in (1,\infty)$, $s\in(1,p)$ and $t\in(1,\infty)$, then every multi-frequency Calder\'on-Zygmund operator $T$ is bounded on $L^p(\omega)$ for every weight $\omega \in RH_{t'} \cap {\mathbb A}_{\frac{p}{s}}$ with
$$ \|T\|_{L^p(\omega) \to L^p(\omega)} \lesssim N^{\gamma}$$
and
$$ \gamma:=\frac{tp}{s\min\{2,s\}}+\left|\frac{1}{2}-\frac{1}{s}\right|.$$
We emphasize that this result is only interesting when $\gamma<1$.
\end{theorem}

\bigskip The current paper is organized as follows: after some preliminaries about weights, examples of multi-frequency operators and the main lemma for the multi-frequency analysis, Theorem \ref{thm:boundedness} is proved in Section \ref{sec1}. Then in Section \ref{sec2}, we develop the general approach for weighted estimates, based on a suitable maximal sharp function. In Section \ref{sec3}, we describe how this point of view could be used to Bochner-Riesz multipliers.

\section{Notations and preliminaries}

Let us consider the Euclidean space $\rn$  equipped with the Lebesgue measure $dx$ and its Euclidean distance $|x-y|$. Given a ball $Q\subset\rn$ we denote its center by $c(Q)$ and its radius by $r_Q$. For any $\lambda>1$, we denote by $\lambda\,Q:=B(c(Q),\lambda r_Q)$.
We write $L^p$ for $L^p(\rn,\re)$ or $L^p(\rn,{\mathbb C})$. For a subset $E\subset \rn$ of finite and non-vanishing measure and $f$ a locally integrable function, the average
of $f$  on $E$ is defined by $$ \aver{E} f dx := \frac{1}{|E|}\int_E f(x) dx.$$
Let us denote by ${\mathcal Q}$ the collection of all balls in $\rn$. We write $\M$ for the maximal Hardy-Littlewood function:
$$
{\mathcal M} f(x)= \sup_{{\genfrac{}{}{0pt}{}{Q\in{\mathcal Q}}{x\in Q}}}  \aver{Q}|f|dx.$$

For $p\in (1,\infty)$, we set $\M_p f(x)=\M(|f|^p)(x)^{1/p}$. The Fourier transform will be denoted by $\mathcal F$ as an operator and we make use of the other usual  notation ${\mathcal F}(f)=\widehat{f}$ too.

\mb

In the current work, we aim to develop a multi-frequency analysis, based on the following lemma:
\begin{lemma}[\cite{BE}] \label{lem} Let $\Theta\subset \R^n$ be a finite collection of frequencies and $Q$ be a ball. For every function $\phi$ belonging to the subspace of $L^2(3Q)$, spanned by $(e^{i\xi\cdot})_{\xi\in \Theta}$, we have for $p\in[1,2]$
\begin{equation}  \|\phi\|_{L^\infty(Q)} \lesssim (\sharp \Theta)^{\frac{1}{p}} \left(\aver{3Q} |\phi|^p dx \right)^{\frac{1}{p}}. \label{aze} \end{equation}
\end{lemma}

\begin{remark} In \cite{BE}, this lemma is stated and proved in a one-dimensional setting. However, the proof only relies on the additive group structure of the ambient space by using translation operators. So the exact same proof can be extended to a multi-dimensional setting.
\end{remark}

\begin{remark} \label{rem:lem} The question of extending the previous lemma for $p\in(2,\infty)$ is still open in such a general situation. Of course, (\ref{aze}) is true for $p=\infty$ and so it would be reasonable to expect the result for intermediate exponents $p \in(2,\infty)$. Unfortunately, the well-known interpolation theory does not apply here. \\
However, in some specific situations, we may extend this lemma for $p\geq 2$. Indeed, if $p=2k$ is an even integer then applying (\ref{aze}) with $p=2$ and $\Theta^k:=\{\theta_{i_1}+...+\theta_{i_k},\ \theta_i\in \Theta\}$ to $\phi^k$ yields
\begin{align*}
 \|\phi \|_{L^\infty(Q)} & \lesssim  \|\phi^k  \|_{L^\infty(Q)}^{\frac{1}{k}}  \\
 &  \lesssim (\sharp \Theta^k )^{\frac{1}{2k}} \left(\aver{3Q} |\phi|^{2k} dx \right)^{\frac{1}{2k}}  \\
 & \simeq (\sharp \Theta^k )^{\frac{1}{p}} \left(\aver{3Q} |\phi|^{p} dx \right)^{\frac{1}{p}}.  
\end{align*}
By this way, we see that an extension of (\ref{aze}) for $p\geq 2$ may be related to sharp combinatorial arguments, to estimate $\sharp \Theta^k$ (a trivial bound is $\sharp \Theta^k \leq (\sharp \Theta)^{k}$ which does not improve (\ref{aze})). 
\end{remark}

We aim to obtain weighted estimates, involving Muckenhoupt's weights. 

\begin{dfn} \label{def:1} A weight $\omega$ is a non-negative locally integrable function.
We say that a weight $\omega\in {\mathbb A}_p$, $1<p<\infty$, if there exists a positive constant $C$ such that for every ball $Q$,
$$\bigg(\aver{Q} \omega\,dx\bigg)\, \bigg(\aver{Q} \omega^{1-p'}\,dx\bigg)^{p-1}\le C.
$$
For $p=1$, we say that $\omega\in {\mathbb A}_1$ if there is a positive constant $C$ such
that for every ball $Q$,
$$
\aver{Q} \omega \,dx
\le
C\, \omega (y),
\qquad \mbox{for a.e. }y\in Q.
$$
We write ${\mathbb A}_\infty=\cup_{p\ge 1} {\mathbb A}_p$.
\end{dfn}

We just recall that for $p\in(1,\infty)$, the maximal function $\M$ is bounded on $L^p(\omega)$ if and only if $\omega \in {\mathbb A}_p$.
We also need to introduce the reverse H\"older classes. 

\begin{dfn} \label{def:2} A weight $\omega \in RH_p$, $1<p<\infty$, if there is a constant $C$ such that for every ball $Q$,
$$
\bigg( \aver{Q} \omega^p\,dx \bigg)^{1/p} \leq C \bigg(\aver{Q} \omega\,dx\bigg).
$$
It is well known that ${\mathbb A}_\infty=\cup_{r>1} RH_r$. Thus, for $p=1$ it is understood that $RH_1={\mathbb A}_\infty$.
\end{dfn}



\subsection{Examples of multi-frequency Calder\'on-Zygmund operators}

Let us detail particular situations where such multi-frequency operators appear.

\subsubsection*{The multi-frequency Hilbert transform}

As explained in the introduction, an example of such multi-frequency operators in the $1$-dimensional setting is the multi-frequency Hilbert transform. In $\R$, consider an arbitrary collection of frequencies $\Theta:=(\xi_1,...,\xi_N)$ (assumed to be indexed by the increasing order $\xi_1< \xi_2< \cdots < \xi_N$). The associated multi-frequency Hilbert transform is the Fourier multiplier corresponding to the symbol
$$ m(\xi) = \left\{ \begin{array}{ll}
                     -1, & \xi<\xi_1 \\
                      (-1)^{j+1}, & \xi_j<\xi<\xi_{j+1} \\
                      (-1)^{N+1}, & \xi>\xi_N.
                    \end{array} \right. $$
Associated to $\Theta$, we have a collection of disjoint intervals $\Delta:=\{ (-\infty,\xi_1), (\xi_1,\xi_2),...,(\xi_N,\infty)\}$. It is well-known by Rubio de Francia's work \cite{Rubio} that for $q\in(1,2]$, the functional 
\begin{equation} f \rightarrow \left( \sum_{\omega\in\Delta} \left|{\mathcal F}^{-1}[{\bf 1}_{\omega} {\mathcal F}f]\right|^q \right)^{\frac{1}{q}} \label{eq:f} \end{equation}
is bounded on $L^p$ for $p\in(q',\infty)$. 

The boundedness of the multi-frequency Hilbert transform is closely related to the understanding of (\ref{eq:f}) for $q\to 1$.

\mb

We point out that in Rubio de Francia's result, the obtained estimates do not depend on the collection of intervals $\Delta$. More precisely, excepted the end-point $p=q'$, the range $(q',\infty)$ is optimal for a uniform (with respect to the collection $\Delta$) $L^p$-boundedness of (\ref{eq:f}). So it is natural that for $q\to 1$ things are more difficult, which is illustrated by our multi-frequency Calder\'on-Zygmund analysis. Indeed, for example if one considers the particular case $\Theta:=(1,...,N)$, then following the notations of Remark \ref{rem:lem}, we have $\Theta^k =\{k,...,kN\}$ and so $\sharp \Theta^k = k(N-1)+1 \simeq k N$. Hence, in this situation we have observed (see Remark \ref{rem:lem}) that we can extend Lemma \ref{lem} to exponents $p\in[1,\infty]$ (the implicit constant appearing in (\ref{aze}) is only depending on $p$). By this way, Theorem \ref{thm:weight} can be improved and we obtain a better exponent
$$ \gamma=\frac{tp}{s^2}+\left|\frac{1}{2}-\frac{1}{s}\right|.$$
Consequently, it seems that for the $L^p$-boundedness of the multi-frequency Hilbert transform, the collection $\Theta$ could play an important role (which was not the case for the $\ell^q$-functional (\ref{eq:f}) with $q'<p$).

\subsubsection*{Multi-frequency operators coming from a covering of the frequency space}

Let $(Q_j)_{j=1,...,N}$ be a family of disjoint cubes and $\phi_j$ a smooth function with $\widehat{\phi_j}$ supported and adapted to $Q_j$. Then consider the linear operator given by
$$ T(f) = \sum_{j=1}^N \phi_j \ast f.$$

It is easy to check that $T$ is a multi-frequency Calder\'on-Zygmund operator, associated to the collection $\Theta:=(\xi_1,...,\xi_N)$ where for every $j$, $\xi_j:=c(Q_j)$ is the center of the ball $Q_j$. With $r_j$ the radius of $Q_j$, we have the regularity estimate
$$ \sum_{j=1}^N \left|\nabla_{(x,y)}\ e^{i\xi_j\cdot(x-y)} \phi_j(x-y)\right| \lesssim |x-y|^{-n-1} \sum_{j=1}^N \frac{(r_j |x-y|)^{n+1}}{(1+r_j|x-y|)^{M}},$$
for every integer $M>0$.

So boundedness of $T$ (Theorem \ref{thm:boundedness}) yields the inequality

\begin{equation} \left\| \sum_{j=1}^N \phi_j \ast f \right\|_{L^p} \lesssim C(r_1,...,r_N) N^{\left|\frac{1}{p}-\frac{1}{2}\right|} \|f\|_{L^p}, \label{eq:ex} \end{equation}
with 
$$ C(r_1,...,r_N):= \sup_{t>0} \sum_{j=1}^N \frac{(r_j t)^{n+1}}{(1+r_jt)^{M}}.$$

Let us examine some particular situations:
\begin{itemize} 
 \item If the cubes $(Q_j)_j$ have an equal side-length, then as for Proposition \ref{prop}, simple arguments imply (\ref{eq:ex}) for $p\in[1,\infty]$ without the constant $C(r_1,...,r_N)$. 
 \item If the collection $(Q_j)_j$ is dyadic: it exists a point $\xi_0$, $d(Q_j,\xi_0) \simeq r_{Q_j} \simeq 2^j$ then Littlewood-Paley theory implies (\ref{eq:ex}) without the factor $N^{|\frac{1}{p}-\frac{1}{2}|}$ (in this case $C(r_1,...,r_N)\simeq 1$).
 \item If the cubes $(Q_j)$ have only the dyadic scale: $r_{Q_j} \simeq 2^j$ (but no assumptions on the centers of the balls) then Littlewood-Paley theory cannot be used. However, our previous results can be applied in this situation and so (\ref{eq:ex}) holds and $C(r_1,...,r_N) \simeq 1$.
\end{itemize}

We aim to use the new multi-frequency Calder\'on-Zygmund analysis to extend these inequalities with replacing the convolution operators by more general Calder\'on-Zygmund operators, still satisfying some orthogonality properties.

\subsubsection*{Multi-frequency operators coming from variation norm estimates}

As explained in the introduction, the multi-frequency Calder\'on-Zygmund analysis has been first developed for proving a variation norm variant of a Bourgain's maximal inequality.
So our results can be adapted in such a framework. For example, in \cite{GMS} Grafakos, Martell and Soria have studied maximal inequalities of the form
$$ \left\| \sup_{j=1,...,N} \left| T( e^{i\theta_j\cdot} f) \right| \right\|_{L^p} \lesssim \|f\|_{L^p}$$
where $(\theta_j)_{j=1,...,N}$ is a collection of frequencies and $T$ a fixed Calder\'on-Zygmund operator.

We can ask the same question, for a variation norm variant: for $q\in[1,\infty)$ consider
 $$ \left(\sum_{j=1}^N \left|  T( e^{i\theta_j\cdot} f) \right|^q \right)^{\frac{1}{q}}$$
and study its boundedness on $L^p$, with a sharp control of the behaviour with respect to $N$. By a linearization argument (involving Rademacher's functions), this $\ell^q$-functional can be realized as an average of modulated Calder\'on-Zygmund operators, associated to the collection $\Theta:=(\theta_j)_j$.

\section{Unweighted estimates for multi-frequency Calder\'on-Zygmund operators} \label{sec1}

In this section, we aim to prove the weak $L^1$-estimate for a multi-frequency Calder\'on-Zygmund operator, then Theorem \ref{thm:boundedness} will easily follow from interpolation and duality.

\begin{proposition} \label{prop:bound} Let $\Theta=(\xi_1,...,\xi_N)$ be a collection of $N$ frequencies as above and $T$ be a Calder\'on-Zygmund operator relatively to $\Theta$. Then $T$ is of weak type $(1,1)$ with (uniformly with respect to $N$)
$$ \|T\|_{L^1\to L^{1,\infty}} \lesssim N^{\frac{1}{2}}.$$
\end{proposition}

\begin{proof} Consider $f$ a function in $L^1$ and $\lambda>0$, we use the Calder\'on-Zygmund decomposition\footnote{In \cite{NOT}, the multi-frequency Calder\'on-Zygmund decomposition is only described in $\R$. The proof is a combination of Lemma \ref{lem} and the usual Calder\'on-Zygmund decomposition. Since both of them can be extended in a multi-dimensional framework, the multi-frequency Calder\'on-Zygmund decomposition performed in \cite{NOT} still holds in $\R^n$.} of \cite{NOT} related to the collection of frequencies $\Theta$. So the function $f$ can be decomposed $f = g + \sum_{J\in {\bf J}} b_J$ with the following properties:
\begin{itemize}
 \item ${\bf J}$ is a collection of balls and $(3J)_{J\in{\bf J}}$ has a bounded overlap;
 \item for each $J\in{\bf J}$, $b_J$ is supported in $3J$;  
 \item we have 
\begin{equation} \sum_{J \in{\bf J}} |J| \lesssim \sqrt{N} \|f\|_{L^1} \lambda^{-1}; \label{eq:sumJ} \end{equation}
\item the ``good part'' $g$ satisfies
\begin{equation} \|g\|_{L^2}^2 \lesssim \|f\|_{L^1} \sqrt{N} \lambda; \label{eq:g} \end{equation}
\item the cubes $J$ satisfy
\begin{equation} \|f\|_{L^1(J)} \lesssim |J|\lambda N^{-\frac{1}{2}} , \quad \|f-b_J\|_{L^2(J)} \lesssim \sqrt{|J|} \lambda; \label{eq:b} \end{equation}
\item we have cancellation for all the frequencies of $\Theta$: for all $j=1,...,N$ and $J\in{\bf J}$, $\widehat{b_J}(\xi_j)=0$.
\end{itemize}
We aim to estimate the measure of the level-set
$$ \Upsilon_\lambda:=\left\{x, |T(f)(x)|>\lambda\right\}.$$
With $b=\sum_J b_J$, we have
\begin{align*}
 |\Upsilon_\lambda| & \leq \left|\left\{x, |T(g)(x)|>\lambda/2\right\}\right| + \left|\left\{x, |T(b)(x)|>\lambda/2\right\}\right| \\
 & \lesssim \lambda^{-2} \|T(g)\|_{L^2}^2 +  \left|\left\{x, |T(b)(x)|>\lambda/2\right\}\right| \\
 & \lesssim \lambda^{-1}\sqrt{N}\|f\|_{L^1} + \left|\left\{x, |T(b)(x)|>\lambda/2\right\}\right|,
\end{align*}
where we used the $L^2$-boundedness of $T$. So it remains us to study the last term.
Since (\ref{eq:sumJ}), we get
$$ \left| \bigcup_{J\in {\bf J}} 4J \right| \lesssim \sum_{J} |J| \lesssim  \sqrt{N} \|f\|_{L^1} \lambda^{-1}.$$
Consequently, it only remains to estimate the measure of the set
$$ O_\lambda:= \left\{x \in \left(\bigcup_{J\in {\bf J}} 4J\right)^{c}, \quad |T(b)(x)|>\lambda/2\right\}.$$
Since
\begin{equation} 
 |O_\lambda| \lesssim \lambda^{-1} \sum_{J} \|T(b_J)\|_{L^1((2J)^c)}, \label{eq:Olambda}
\end{equation}
it is sufficient to estimate the $L^1$-norms.
Consider $K$ the kernel of $T$ and a point $x_0\in \left(\bigcup_{J\in {\bf J}} 4J\right)^{c}$.
Then, we can use the integral representation and we have
$$ T(b)(x_0) = \int K(x_0,y) b(y) dy = \sum_{J} \int_{3J} K(x_0,y) b_J(y) dy.$$
To each $J$, we aim to take advantage of the cancellation properties of $b_J$, so we subtract the projection of $\left[y\to K(x_0,y)\right]$ on the space, spanned by $(e^{iy\cdot\eta})_{\eta\in \Theta}$. So we have
\begin{align*}
 T(b)(x_0) & = \sum_{J}\sum_{j=1}^N  \int_{3J} \left[K_j(x_0,y) - e^{i\xi_j \cdot c(J)} K_j(x_0,c(J))e^{-i\xi_j \cdot y}\right] b_J(y) dy \\
 & = \sum_{J}\sum_{j=1}^N  \int_{3J} \left[\widetilde{K}_{j}(x_0,y) - \widetilde{K}_j(x_0,c(J))\right]e^{i\xi_j\cdot (x_0-y)} b_J(y) dy
\end{align*}
where $c(J)$ is the center of $J$ and $\widetilde{K}_j(x,y):=K_j(x,y)e^{-i\xi_j \cdot (x-y)}$.
We then write 
$$ T_j(b)(x_0) :=  \int \left[\widetilde{K}_{j}(x_0,y) - \widetilde{K}_j(x_0,c(J))\right]e^{i\xi_j \cdot (x_0-y)} b(y) dy.$$
such that $T(b) =\sum_j T_j(b)$. Due to the regularity assumption on $K$ (and so on $\widetilde{K}_j$), it comes for $y\in J$ and $x_0\in (2J)^c$
\begin{equation} 
\sum_{j=1}^N \left|\widetilde{K}_{j}(x_0,y) - \widetilde{K}_j(x_0,c(J))\right| \lesssim \frac{r_J}{|x_0-y|^{n+1}}. \label{eq:kernelbis} \end{equation}
So we have
$$ \|T(b_J)\|_{L^1((2J)^c)} \lesssim \iint_{|x-y|\geq r_J} \frac{r_J}{|x-y|^{n+1}} |b_J(y)| dxdy \lesssim \|b_J\|_{L^1} \lesssim |J| \lambda.$$
Finally, we obtain with (\ref{eq:Olambda}) that
$$ |O_\lambda| \lesssim \sum_{J} |J| \lesssim  \sqrt{N} \|f\|_{L^1} \lambda^{-1}, $$
which concludes the proof.
\end{proof}

\begin{remark}
 Following \cite{NOT}, the bound of order $N^{\frac{1}{2}}$ is optimal for the multi-frequency decomposition and for the weak-$L^1$ estimate.
\end{remark}

\section{Weighted estimates for multi-frequency Calder\'on-Zygmund operators} \label{sec2}

Aiming to obtain weighted estimates on such multi-frequency operators (using {\it Good-lambda inequalities}), we also have to define a suitable maximal sharp function, associated to a collection of frequencies.

\begin{dfn}[Maximal sharp function] \label{def:max} Let $\Theta$ be a collection of $N$ frequencies and $s\in[1,\infty)$. Consider a ball $Q$, we denote by $\pr_{\Theta,Q}$ the projection operator (in the $L^s$-sense) on the subspace of $L^s(3Q)$, spanned by $(\exp {i\xi\cdot})_{\xi\in \Theta}$. Let us specify this projection operator: consider $E$ the finite dimensional sub-space of $L^s(3Q)$, spanned by $(e^{i\xi\cdot})_{\xi\in \Theta}$ and equipped with the $L^s(3Q)$-norm.
Since $E$ is of finite dimension, then for every $f\in L^s(Q)$ there exists $v:=\pr_{\Theta,Q}(f)\in E$ such that
$$ \|f-v\|_{L^s(3Q)} = \inf_{\phi\in E} \|f-\phi\|_{L^s(3Q)}.$$
This projection operator may depend on $s$, which is not important for our purpose so this is implicit in the notation and we forget it.

\mb
Since $0\in E$, we obviously have 
\begin{equation} \label{eq:projection}\|\pr_{\Theta,Q}(f)\|_{L^s(3Q)} \leq 2\|f\|_{L^s(Q)}. \end{equation} Then, we may define the maximal sharp function 
$$ \M_{s,\Theta}^\sharp (f)(x_0) := \sup_{x_0\in Q} \left(\aver{Q} \left|f-\pr_{\Theta,Q}(f{\bf 1}_Q) \right|^s  dx \right)^{\frac{1}{s}}.$$
\end{dfn}

Note that the usual sharp maximal function is the one obtained for $\Theta:=\{0\}$ and in this situation it is well-known that the maximal sharp function satisfies a so-called Fefferman-Stein inequality (see \cite{FS}). We first prove an equivalent property for this generalised maximal sharp function:

\begin{proposition} \label{prop:FeffermanStein} Let $s\in(1,\infty)$, $t\in[1,\infty)$ and $p \in(s,\infty)$ be fixed. Then for every function $f\in L^s$ and every weight $\omega\in RH_{t'}$, we have for every $p\geq s$
$$ \|f\|_{L^p(\omega)} \lesssim N^{\frac{tp}{s}\max\{\frac{1}{2},\frac{1}{s}\}} \left\| \M_{s,\Theta}^\sharp(f) \right\|_{L^p(\omega)}.$$
\end{proposition}

The proof relies on a {\it Good-lambda inequality} and Lemma \ref{lem}.

\begin{proof} We make use on the abstract theory developed in \cite{AM} by Auscher and Martell.
We also follow notations of \cite[Theorem 3.1]{AM}. Indeed, for each ball $Q\subset \rn$ we have the following 
 $$ F(x) := |f(x)|^s \lesssim \left|f(x)-\pr_{\Theta,Q}(f{\bf 1}_Q)(x) \right|^s + \left|\pr_{\Theta,Q}(f{\bf 1}_Q)(x) \right|^s := G_Q(x)+H_Q(x).
$$ 
 By definition, it comes
 $$ \aver{Q} G_Q dx \leq \inf _{Q} \M_{s,\Theta}^\sharp (f)^s$$
 and following Lemma \ref{lem} (with (\ref{eq:projection}))
 \begin{align*}
  \sup_{x\in Q} H_Q & = \|\pr_{\Theta,Q}(f {\bf 1}_Q)\|_{L^\infty(Q)}^s \lesssim N^{s\max\{\frac{1}{2},\frac{1}{s}\}} \left( \aver{3Q} |\pr_{\Theta,Q}(f {\bf 1}_Q)|^s dx\right) \\
  & \lesssim N^{s\max\{\frac{1}{2},\frac{1}{s}\}} \left( \aver{Q} |f|^s dx\right) \lesssim  N^{s\max\{\frac{1}{2},\frac{1}{s}\}} \inf_Q {\mathcal M} F.
  \end{align*}
 So we can apply \cite[Theorem 3.1]{AM} (with $q=\infty$ and $a\simeq N^{s\max\{\frac{1}{2},\frac{1}{s}\}}$) and by checking the behaviour of the constants with respect to ``$a$" in its proof, we obtain for every $p\geq 1$
 $$ \left\| {\mathcal M}_s(f)^s \right\|_{L^p(\omega)} \lesssim N^{spt \max\{\frac{1}{2},\frac{1}{s}\}} \left\| \M_{s,\Theta}^\sharp (f)^s \right\|_{L^p(\omega)},$$
 which yields the desired result.
\end{proof}

Then, we evaluate a multi-frequency Calder\'on-Zygmund operator via this new maximal sharp function.

\begin{proposition} \label{prop:maximal} Let $T$ be a Calder\'on-Zygmund operator relatively to $\Theta$ and $s\in(1,\infty)$. Then, we have the following pointwise estimate:
$$  \M_{s,\Theta}^\sharp (T(f)) \lesssim N^{|\frac{1}{s}-\frac{1}{2}|} {\mathcal M}_s(f) .$$
\end{proposition}

\begin{proof} We follow the well-known proof for usual Calder\'on-Zygmund operators and adapt the arguments to the current situation. So consider a point $x_0$ and a ball $Q\subset \rn$ containing $x_0$, we have to estimate
$$ \left( \aver{Q} \left|T(f)-\pr_{\Theta,Q}(T(f){\bf 1}_Q) \right|^s dx \right)^{\frac{1}{s}}.$$
We split the function into a local part $f_0$ and an off-diagonal part $f_\infty$: 
$$f=f_0+f_\infty:= f{\bf 1}_{10 Q} + f {\bf 1}_{(10 Q)^c}.$$ 
By definition of the projection operator, we know that
\begin{align*}
 \left( \aver{Q} \left|T(f)-\pr_{\Theta,Q}(T(f){\bf 1}_Q) \right|^s dx \right)^{\frac{1}{s}} & \leq \left( \aver{Q} \left|T(f)-\pr_{\Theta,Q}(T(f_\infty){\bf 1}_Q) \right|^s dx \right)^{\frac{1}{s}} \\
 & \leq \left( \aver{Q} \left|T(f_0) \right|^s dx \right)^{\frac{1}{s}} + \left( \aver{Q} \left|T(f_\infty)-\pr_{\Theta,Q}(T(f_\infty){\bf 1}_Q) \right|^s dx \right)^{\frac{1}{s}}.
\end{align*}
For the local part, we use boundedness in $L^s$ of the operator $T$ (Proposition \ref{prop:bound}), hence
\begin{align*}
 \left( \aver{Q} \left|T(f_0)\right|^s dx \right)^{\frac{1}{s}} & \lesssim |Q|^{-\frac{1}{s}} \|T(f_0)\|_{L^s(Q)} \lesssim  N^{(\frac{1}{2}-\frac{1}{s})} \left(|Q|^{-\frac{1}{s}} \|f_0\|_{L^s}\right) \\
& \lesssim N^{|\frac{1}{2}-\frac{1}{s}|} {\mathcal M}_s(f)(x_0).
\end{align*}
Then let us focus on the second part, involving $f_\infty$. \\
We use the decomposition (with an integral representation) since we are in the off-diagonal case: for $x\in Q$ 
$$ T(f_\infty)(x) = \sum_{j=1}^N  \int K_j(x,y) f_\infty(y) dy.$$
Consider the following function, defined on $3Q$ by (where $c(Q)$ is the center of $Q$)
$$ \Phi:=  x\in 3Q \rightarrow \sum_{j=1}^N  \int  e^{i\xi_j\cdot (x-c(Q))} K_j(c(Q),y) f_\infty(y) dy.$$
So $\Phi \in E$ (see Definition \ref{def:max}) and hence 
\begin{equation} \left( \aver{Q} \left|T(f_\infty)-\pr_{\Theta,Q}(T(f_\infty){\bf 1}_Q) \right|^s dx \right)^{\frac{1}{s}} \leq \left( \aver{Q} \left|T(f_\infty)-\Phi \right|^s dx \right)^{\frac{1}{s}}. \label{eq::} \end{equation}
If we set $\widetilde{K}_j(x,z):=K_j(x,z)e^{-i\xi_j\cdot (x-z)}$, then 
$$ T(f_\infty)(x) -\Phi(x) =  \sum_j \int \left[\widetilde{K}_{j}(x,y) - \widetilde{K}_j(c(Q),y)\right] e^{i\xi_j(x-y)} f_\infty(y) dy.$$
From the regularity assumption on the kernels $K_j$'s, we have for $y\in (10 Q)^c$
\begin{equation} \sum_{j} \left| \widetilde{K}_{j}(x,y) - \widetilde{K}_j(c(Q),y) \right| \lesssim r_Q \sup_{z\in Q} \sum_j \left| \nabla_{x} \widetilde{K}_{j}(z,y)\right| \lesssim r_Q^{-n}\left(1+\frac{d(y,Q)}{r_Q}\right)^{-n-1}. \label{eq:kernel} \end{equation}
We also have (since $y\in (10 Q)^c$ and $x,c(Q)\in Q$)
\begin{align*}
 \left| T(f_\infty)(x) -\Phi(x)\right| & \lesssim  \int_{|z|\geq 10 r_Q} r_Q^{-n}\left(1+\frac{|x-c(Q)-z|}{r_Q}\right)^{-n-1} |f(c(Q)+z)| dz \\
 & \lesssim \int_{|z|\geq 5 r_Q} r_Q^{-n}\left(1+\frac{|z|}{r_Q}\right)^{-n-1} |f(x_0+z)| dz \\
 & \lesssim {\mathcal M}(f)(x_0),
\end{align*}
which concludes the proof.
\end{proof}

We obtain the following corollary:

\begin{corollary} \label{cor} Let $\Theta$ be a collection of $N$ frequencies. For $p\in (2,\infty)$, $s\in[2,p)$ and $t\in(1,\infty)$, a multi-frequency Calder\'on-Zygmund operator $T$ is bounded on $L^p(\omega)$ for every weight $\omega \in RH_{t'} \cap {\mathbb A}_{\frac{p}{s}}$ with
$$ \|T\|_{L^p(\omega) \to L^p(\omega)} \lesssim N^{\frac{tp}{2s}+\left(\frac{1}{2}-\frac{1}{s}\right)}.$$
\end{corollary}

\begin{proof} Using Propositions \ref{prop:FeffermanStein} and \ref{prop:maximal}, it follows that for $p>s\geq 2$ (assuming $\omega \in {\mathbb A}_{\frac{p}{s}}$)
\begin{align*} \|T(f) \|_{L^p(\omega)} & \lesssim N^{\frac{tp}{2s}} \left\| \M_{s,\Theta}^\sharp [T(f)] \right\|_{L^p(\omega)} \\
& \lesssim N^{\frac{tp}{2s}+\left(\frac{1}{2}-\frac{1}{s}\right)}  \left\| {\mathcal M}_s(f) \right\|_{L^p(\omega)} \\
& \lesssim N^{\frac{tp}{2s}+\left(\frac{1}{2}-\frac{1}{s}\right)}  \left\| f \right\|_{L^p(\omega)},
\end{align*}
where we used weighted boundedness of the maximal function since $\omega \in {\mathbb A}_{\frac{p}{s}}$.
\end{proof}

\mb As explained in the introduction, this estimate is only interesting when the exponent $\frac{tp}{2s}+\left(\frac{1}{2}-\frac{1}{s}\right)$ is lower than $1$.

\section{Connexion to Bochner-Riesz multipliers} \label{sec3}

In this section, we aim to describe how such arguments could be applied to generalized Bochner-Riesz multipliers. Weighted estimates for Bochner-Riesz multipliers has been initiated in \cite{Vargas, Christ, CDL}. We first emphasize that we do not pretend to obtain new weighted estimates for Bochner-Riesz multipliers. But we only want to describe here a new point of view and a new approach for such estimates, which will be the subject of a future investigation. Such an application is a great motivation for pursuing the study of a multi-frequency Calder\'on-Zygmund analysis.

\mb

Consider also $\Omega$ a bounded open subset of $\rn$ such that its boundary $\Gamma:=\overline{\Omega}\setminus \Omega$ is an hyper-manifold of Hausdorff dimension $n-1$. For $\delta>0$, we then define the generalized Bochner-Riesz multiplier, given by
$$ R_{\Omega,\delta}(f) (x) := \int_\Omega e^{ix\cdot \xi} \widehat{f}(\xi) m_\delta d\xi,$$
where $m_\delta$ is a smooth symbol supported in $\overline{\Omega}$ and satisfying in $\Omega$
$$ | \partial^\alpha m_\delta(\xi)| \lesssim d(\xi,\Gamma)^{\delta-|\alpha|}.$$

\mb 
We first use a Whitney covering $(O_i)_i$ of $\Omega$. That is a collection of sub-balls such that
\begin{itemize}
 \item the collection $(O_i)_i$ covers $\Omega$ and has a bounded overlap;
 \item the radius $r_{O_i}$ is equivalent to $d(O_i,\Gamma)$.
\end{itemize}
Associated to this collection, we build a partition of the unity $(\chi_i)_i$ of smooth functions such that $\chi_i$ is supported on $O_i$ with 
$$ \sum_i \chi_i(\xi) = {\bf 1}_{\Omega}(\xi)$$
and $\|\partial^\alpha \chi_i\|_\infty \lesssim r_{O_i}^{-|\alpha|}$. \\
Then, $R_{\delta}$ may be written as
$$ R_{\delta}(f) (x) = \sum_{j=-\infty}^\infty T_j(f)(x),$$
with
\begin{align}
 T_j(f)(x) & :=\sum_{\genfrac{}{}{0pt}{}{l,}{2^j \leq r_{O_l}< 2^{j+1}}} \int_\Omega e^{ix\cdot \xi} \widehat{f}(\xi) m_\delta(\xi) \chi_l(\xi) d\xi \nonumber \\
 & =2^{j\delta} U_j(f)(x), \label{eq:dec}
 \end{align}
where we set
$$ U_j(f)(x):= \sum_{\genfrac{}{}{0pt}{}{l,}{2^j \leq r_{O_l}< 2^{j+1}}} \int_\Omega e^{ix\cdot \xi} \widehat{f}(\xi) (2^{-j\delta} m_\delta(\xi)) \chi_l(\xi) d\xi.$$

\gb
{\bf Observation :} The main idea is to observe that the operator $U_j$ is a multi-frequency Calder\'on-Zygmund operator associated to the collection 
$$ \Theta_j :=\{ c(O_l), \ 2^j \leq r_{O_l}< 2^{j+1} \} \qquad \textrm{with} \qquad \sharp \Theta_j \simeq 2^{-j(n-1)}.$$

However, these operators have specific properties, one of them is that the considered balls have equivalent radius, which means that these operators have only one scale $2^{j}$. For example, this observation allows us to easily prove some boundedness:

\begin{proposition} \label{prop} Uniformly with $j\lesssim 0$, the multiplier $U_j$ is a convolution operation with a kernel $K_j$ satisfying
$$ \|K_j\|_{L^1} \lesssim 2^{-j \frac{n-1}{2}}.$$
Hence, it follows that $U_j$ is bounded on Lebesgue space $L^p$ for every $p\in[1,\infty]$. Moreover for every $s\in[1,2]$, $p\in(s,\infty)$ and every weight $\omega\in {\mathbb A}_{\frac{p}{s}}$, $U_j$ is bounded on $L^p(\omega)$ with
$$ \| U_j\|_{L^p(\omega) \rightarrow L^p(\omega)} \lesssim 2^{-j\frac{n-1}{s}}.$$
\end{proposition}

\begin{proof} The operator $U_j$ is a Fourier multiplier, associated to the symbol 
$$\sigma_j(\xi):=\sum_{\genfrac{}{}{0pt}{}{l,}{2^j \leq r_{O_l}< 2^{j+1}}} (2^{-j\delta} m_\delta(\xi)) \chi_l(\xi).$$ 
Since the considered balls $(O_l)_l$ are almost disjoint, it comes that 
$$ \|\sigma_j\|_{L^2} \lesssim |\{\xi, d(\xi,\partial \Omega)\simeq 2^j \}|^{\frac{1}{2}} \lesssim 2^{\frac{j}{2}}.$$
Moreover, using regularity assumptions on $m_\delta$, we deduce that for every $\alpha$
$$ \|\partial^\alpha \sigma_j\|_{L^2} \lesssim 2^{-j|\alpha|} |\{\xi, d(\xi,\partial \Omega) \simeq 2^j \}|^{\frac{1}{2}} \lesssim 2^{j(\frac{1}{2}-|\alpha|)}.$$
So with $K_j:= \mathcal{F}(\sigma_j)$, it follows that for any integer $M$
\begin{equation} \left\| (1+2^{j}|\cdot|)^{M} K_j\right\|_{L^2} \lesssim 2^{\frac{j}{2}}. \label{eq:L^2} \end{equation}
Hence
$$ \left\| K_j\right\|_{L^1} \lesssim 2^{-j\frac{n-1}{2}}.$$
Using Minkowski inequality, we deduce that for every $p\in[1,\infty]$
$$ \| U_j\|_{L^p \to L^p} \lesssim \|K_j\|_{L^1} \lesssim  2^{-j\frac{n-1}{2}}.$$
Let us now focus on the second claim about weighted estimates. Using integrations by parts for computing the kernel $K_j$, it comes for any integer $M$
\begin{equation} \left\| (1+2^{j}|\cdot|)^{M} K_j\right\|_{L^\infty} \lesssim 2^{j}. \end{equation}
By interpolation with (\ref{eq:L^2}), for $s\in[1,2]$ we get
\begin{equation} \left\| (1+2^{j}|\cdot|)^{M} K_j\right\|_{L^{s'}} \lesssim 2^{\frac{j}{s}}, \end{equation}
which gives
$$  U_j(f) \lesssim 2^{-j\frac{n-1}{s}} \M_{s}(f).$$
Hence, for every $p> s$ and every weight $\omega \in {\mathbb A}_{\frac{p}{s}}$
$$ \|U_j\|_{L^p(\omega) \to L^p(\omega)} \lesssim 2^{-j\frac{n-1}{s}}.$$
\end{proof}

\mb
In this context, $\sharp \Theta_j \simeq 2^{-j(n-1)}$, so the constant $2^{-j \frac{n-1}{s}}$ is equivalent to $(\sharp \Theta_j)^{\frac{1}{s}}$ and this is a better constant than the one obtained in Corollary \ref{cor} (for a subclass of ${\mathbb A}_{\frac{p}{s}}$ weights). \\
So improving these ``easy bounds'' means to obtain inequalities such as
$$ \| U_j\|_{L^p(\omega) \rightarrow L^p(\omega)} \lesssim (\sharp \Theta_j)^{\gamma}$$
for some better exponent $\gamma<\frac{1}{s}$.

\gb

Let us finish by suggesting how could we get improvements of our approach to get interesting results for Bochner-Riesz multipliers:

\mb
{\bf Question :} The general approach, developed in the previous section, only allows to get an exponent
$$ \gamma=\frac{tp}{2s}+\left(\frac{1}{2}-\frac{1}{s}\right)$$
(with some $s\in[2,p)$) which is bigger than $\frac{1}{2}$ (since $p>s\geq 2$ and $t>1$). 
So to improve this exponent $\gamma$, two things seem to be crucial: 
\begin{itemize}
 \item to extend the use of Lemma \ref{lem} for $p\geq 2$ which would allow us to get an exponent $\frac{tp}{s^2}$ instead of $\frac{tp}{2s}$;
 \item to use the geometry of the boundary $\Gamma$ to get better exponents, even for the unweighted estimates. Indeed, for example for the unit ball (using its non vanishing curvature), we know that (see \cite{Lee, survey})
$$ \|U_j\|_{L^p\to L^p} \lesssim 2^{-j\delta(p)}$$
with if $n=2$
$$ \delta(p):= \max\left\{ 2\left|\frac{1}{2}-\frac{1}{p}\right|-\frac{1}{2},0 \right\}.$$
and if $n\geq 3$ and $p\geq \frac{2(n+2)}{n}$ or $p\leq \frac{2(n+2)}{n+4}$
$$ \delta(p):= \max\left\{ n\left|\frac{1}{2}-\frac{1}{p}\right|-\frac{1}{2},0 \right\}.$$
\end{itemize}

\end{document}